\documentclass{article}

\usepackage{amsmath,latexsym,amssymb,amsfonts,amsthm}
\usepackage{empheq}
\input xy
\xyoption{all}

\newlength{\defbaselineskip}
\newcommand{\setlinespacing}[1]%
           {\setlength{\baselineskip}{#1 \defbaselineskip}}

\newcommand{\C}{\mathbf{C}}

\newtheorem{definition}{Definition}

\newtheorem{theorem}{Theorem}
\newtheorem{proposition}{Proposition}

\newtheorem{remark}{Remark}
\newtheorem{corollary}{Corollary}
\newtheorem{condition}{Condition}

\makeatletter \@addtoreset{equation}{section} \makeatother
\author{Tamar Janelidze-Gray}
\title{A new approach to $\mathbf{S}$-protomodular categories}

\begin{document}
\maketitle

\begin{abstract}We propose a new approach to $\mathbf{S}$-protomodular categories in the sense of D.\ Bourn, N.\ Martins-Ferreira, A.\ Montoli, and M.\ Sobral. Instead of points (=split epimorphisms) it uses generalized points, which we define as composable pairs of morphisms whose composites are pullback stable regular epimorphisms. This approach is convenient in describing the connection between split and regular Schreier epimorphisms of monoids.
\end{abstract}

\section{Introduction}
An $\mathbf{S}$\textit{-protomodular category} in the sense of D. Bourn, N. Martins-Ferreira, A. Montoli, and M. Sobral (see [3] and [5]) is a pointed category $\mathbf{C}$ (with finite limits) equipped with a class of distinguished split epimorphisms (=``points'') satisfying certain conditions; we will recall the precise definition in Section 3. The introduction to [5] says (among many other things):\\

``Note that our approach to relative non-abelian homological algebra is different from the one initiated by T. Janelidze in \cite{[J]} and developed by her in several later papers: in our work, the word ``relative'' refers to a chosen class of points, i.e. of split epimorphisms with specified splitting, while in T. Janelidze's papers it refers to a chosen class of (not necessarily split) regular epimorphisms.''\\      

Following this remark, in a sense, in this paper we develop what might be called \textit{non-split (epimorphism) approach to the theory of $\mathbf{S}$-protomodular categories}. We replace
\textit{points} $(f,s)=$
\begin{equation*}
\xymatrix{A\ar@<0.5ex>[r]^-f&B,\,\,fs=1_B\ar@<0.5ex>[l]^-s}
\end{equation*}
with pairs $(f,g)=$
\begin{equation*}
\xymatrix{C\ar[r]^g&A\ar[r]^f&B}
\end{equation*}
where the composite $fg$ is required to be a pullback stable regular epimorphism, which we call \textit{generalized points}. We impose certain conditions on a class $\mathbf{T}$ of generalized points in a category $\mathbf{C}$, and call $\mathbf{C}$ a $\mathbf{T}$\textit{-protomodular category} when those conditions are satisfied. Our main result (Theorem 3.4) describes, for an arbitrary pointed category $\mathbf{C}$ with finite limits, a bijection between the collections of classes $\mathbf{S}$ making $\mathbf{C}$ an $\mathbf{S}$-protomodular category and the collections of classes $\mathbf{T}$ making $\mathbf{C}$ a $\mathbf{T}$-protomodular category. In this formulation we have in mind of course that $\mathbf{S}$ always denotes a class of points while $\mathbf{T}$ always denotes a class of generalized points.

In the last section we take $\mathbf{C}$ to be the category of monoids and prove Theorems 4.5 and 4.6, which show that using generalized points is convenient in describing the connection between split and regular Schreier epimorphisms of monoids.

A remark on terminology: The terms \textit{strong (point)} and \textit{regular (Schreier epimorphism)} we are using might sound confusing since there more standard and familiar \textit{strong epimorphisms} and \textit{regular epimorphisms} in general category theory. But this choice of terminology was already made several years ago in the literature we refer to.

\section{Generalized points}
Throughout this paper we assume that $\C$ is a pointed category with finite limits.
\begin{definition} A generalized point in $\C$ is a pair $(f,g)$, in which $f:A\rightarrow B$ and $g:C\rightarrow A$ are morphisms  in $\C$ such that the composite $fg:C\rightarrow B$ is a pullback stable regular epimorphism in $\C$.
\end{definition}
We will also write
\begin{equation}\label{diagram GP}\vcenter{\xymatrix{
		C\ar[r]^g & A \ar[r]^{f} & B}} 
\end{equation}
for the generalized point $(f,g)$ in $\C$. Note that since $fg$ is a pullback stable regular epimorphism, it follows that $f$ is also a pullback stable regular epimorphism (see Proposition 1.5 of \cite{[JST]}).


A morphism between two such generalized points $(f,g)$ and $(f',g')$ is a triple $(\alpha , \beta , \gamma)$, in which $\alpha : A\rightarrow A'$, $\beta : B\rightarrow B'$, and $\gamma : C\rightarrow C'$ are morphisms in $\C$ such that the diagram
\begin{equation}\label{Diagram morphism of GP}\vcenter{ \xymatrix@C=3pc@R=3pc{
C \ar[r]^g \ar[d]_{\gamma} & A\ar[d]_{\alpha} \ar[r]^-f & B \ar[d]^{\beta} \\
C' \ar[r]_{g'} & A'\ar[r]_-{f'} & B' }}
\end{equation}
commutes. 

We will denote by $\mathbf{GPt(\C)}$ the category of all generalized points in $\C$.
Note that if we take all those generalized points (\ref{diagram GP}) for which the composite $fg$ is the identity, then we obtain the category of points $\mathbf{Pt}(\C)$ in $\C$ in the sense of D. Bourn \cite{[B]}. We will also call such generalized points \textit{split generalized points.}


\begin{definition}\label{def strong generalized points} A generalized point $(f,g)$ is said to be a strong generalized point, if $g$ and the kernel of $f$ are jointly strongly epic. 
\end{definition}
In particular, a split generalized point that is strong is the same as a strong split epimorphism in the sense of \cite{[B1]}, and the same as a regular point in the sense of \cite{[MMS]}, and a strong point in the sense of \cite{[BMMS]}.

Let $(f,g)$ be a generalized point in $\C$ and let $x:X\rightarrow B$ be any morphism in $\C$. Taking the pullbacks $(A\times_BX, \pi_1,\pi_2)$ of $f$ and $x$, and $(C\times_BX, {\pi'_1},{\pi'_2})$ of $fg$ and $x$, we obtain the commutative diagram 
\begin{equation}\label{lemma diagram}\vcenter{
	\xymatrix@C=3pc{
	C\times_BX \ar@{.>}[r]^-{g\times1} \ar[d]_{{\pi'_1}}&	A\times_BX \ar[d]_{\pi_1} \ar[r]^-{\pi_2} & X\ar[d]^x \\
	C\ar[r]_g	&A \ar[r]_{f}& B} }
\end{equation}
in which $\pi_2(g\times1)=\pi'_2$, and not only the second square but also the first one is a pullback. Since $(f,g)$ is a generalized point, so is also the pair $(\pi_2,g\times1)$, and we informally say that $(\pi_2,g\times1)$ is the pullback of $(f,g)$ along $x$ in $\mathbf{GPt(\C)}$.


We have:
\begin{theorem}\label{thm pb reflects jse}	If in the commutative diagram (\ref{lemma diagram}), $(f,g)$ is a generalized point in $\C$, $x$ is a pullback stable regular epimorphism, and $\mathrm{ker}(\pi_2)$ and $g\times1$ are jointly strongly epic, then $\mathrm{ker}(f)$ and $g$ are also jointly strongly epic.
\end{theorem}
\begin{proof}Let $m:M\rightarrow A$ be a monomorphism, and let $u:\mathrm{Ker}(f)\rightarrow M$ and $v:C\rightarrow M$ be any morphisms with $mu=\mathrm{ker}(f)$ and $mv=g$. We have to prove that, under the assumptions of the theorem, $m$ is an isomorphism. For, consider the diagram
\begin{equation}\label{D proof of Thm pullback stable strong}\vcenter{
\xymatrix@C=3ex@R=5ex{C\times_BX\ar@{.>}[rd]^{1\times v}\ar[rrr]^{g\times1} \ar[dddd]_{\pi'_1}&&&A\times_BX\ar[rrrr]^{\pi_2}\ar[dddd]^{\pi_1}&&&&X\ar[dddd]^x\\
	&M\times_BX\ar[urr]^{m\times1}\ar[dd]_{\pi''_1}\\
	&&K\ar@{.>}[ul]_-{u'}\ar[uur]^{k'}\ar[dl]^u\ar[ddr]_{k}\\
	&M\ar[drr]_m\\
	C\ar[ur]^v\ar[rrr]_g&&&A\ar[rrrr]_f&&&&B}
 }
\end{equation}
in $\C$, which adds more arrows to diagram (\ref{lemma diagram}) as follows:
\begin{itemize}
\item Since the right-hand square of (\ref{D proof of Thm pullback stable strong}) is a pullback, we can identify $\mathrm{Ker}(\pi_2)$ with $\mathrm{Ker}(f)$, and this object is denoted by $K$; accordingly, we can write $(K,k)=(\mathrm{Ker}(f),\mathrm{ker}(f))$ and $(K,k')=(\mathrm{Ker}(\pi_2),\mathrm{Ker}(\pi_2))$.
\item $m$, $u$, and $v$ are as above.
\item $M\times_BX$ is obtained as the pullback of $fm$ and $x$; its first pullback projection is denoted by $\pi''_1$, while the second one coincides with the composite of $m\times1:M\times_BX\rightarrow A\times_BX$ and $\pi_2:A\times_BX\rightarrow X$. It follows that $(M\times_BX, \pi''_1, m\times1)$ is the pullback of $m$ and $\pi_1$, therefore $m\times1$ is a monomorphism since so is $m$. Note also that $\pi_1$ is a pullback stable regular epimorphism since so is $x$.

\item Since $mu=k=\pi_1k'$, there exists a (unique) morphism $u':K\rightarrow M\times_BX$ such that $(m\times 1)u'=k'$ (and also $\pi''_1u'=u$).

\end{itemize}
Since $k'=\mathrm{ker}(\pi_2)$ and $g\times1$ are jointly strongly epic, it follows that $m\times1$ is an isomorphism. Therefore, since $\pi_1$ is a pullback stable regular epimorphism and $(M\times_BX, \pi''_1, m\times1)$ is the pullback of $m$ and $\pi_1$, it follows that $m$ is an isomorphism (see e.g.\ Proposition 1.6 of \cite{[JST]}).
\end{proof}
Note the following simple fact:

\begin{proposition}\label{prop for them}In a commutative diagram
$$
\xymatrix@C=3pc@R=3pc{& B \ar[d]_s \ar[dl]_h \ar[dr]^{1_B} \\
	               C \ar[r]_g & A \ar[r]_f & B}
$$	
in $\C$, if $\mathrm{ker}(f)$ and $s$ are jointly strongly epic, then $(f,g)$ is a generalized strong point in $\C$. 
\end{proposition}

Given a generalized point $(f,g)$ as in (\ref{lemma diagram}), consider the commutative diagram
$$
\xymatrix@C=3pc@R=3pc{& C \ar[dl]_{\langle 1_C,1_C\rangle} \ar[d]_>>>>>{\langle g,1_C\rangle} \ar[dr]^{1_C} \\
	C\times_BC \ar[r]_-{g\times1} \ar[d]_{\pi'_1}&	A\times_BC \ar[d]_{\pi_1} \ar[r]_-{\pi_2} & C\ar[d]^{fg}\\
	C\ar[r]_g	&A \ar[r]_{f}& B}
$$
whose bottom part is the same as diagram (\ref{lemma diagram}) with $x=fg$. From Theorem \ref{thm pb reflects jse} and Proposition \ref{prop for them}, we obtain:

\begin{corollary} \label{Corollary 1}In the notation above, if $\langle g,1_C\rangle$ and $\mathrm{ker}(\pi_2)$ are jointly strongly epic, then $(f,g)$ is a generalized strong point. 
\end{corollary}

\section{T-protomodular categories}
Let $\mathbf{S}$ and $\mathbf{T}$ be classes of points ($=$ split generalized points) and of generalized points, respectively, in $\mathbf{C}$. Consider the following conditions:

\begin{condition}\label{Condition S}
	The class $\mathbf{S}$ has the following properties: 
	\begin{itemize}
		\item [(a)] it is pullback stable;
		\item [(b)] it is closed under finite limits in $\mathbf{Pt}(\mathbf{C})$;
		\item [(c)] all its elements are strong.
	\end{itemize}
\end{condition}
\begin{condition}\label{Condition T}
	The class $\mathbf{T}$ has the following properties: 
	\begin{itemize}
		\item [(a)] it is pullback stable;
		\item [(b)] it is closed under component-wise finite limits in $\mathbf{GPt}(\mathbf{C})$;
		\item [(c)] all its elements are strong;
		\item [(d)] in a diagram of the form
		\begin{equation*}
		\xymatrix{C\ar@{=}[d]\ar[r]^-{\langle g,1_{C}\rangle}&A\times_BC\ar[d]_{\pi_1}\ar[r]^-{\pi_2}&C\ar[d]^{fg}\\C\ar[r]_g&A\ar[r]_f&B,}
		\end{equation*}
		where $(f,g)$ is a generalized point and the right-hand square is a pullback, the top row belongs to $\mathbf{T}$ if and only if the bottom row does.
	\end{itemize}
\end{condition}
According to \cite{[BMMS]} (originally from \cite{[BMMS2]}), to say that $\mathbf{C}$ is $\mathbf{S}$-\textit{protomodular} is to say that $\mathbf{S}$ satisfies Condition \ref{Condition S}, and we introduce:
\begin{definition}\label{Def T-protomodular category}
	We say that $\mathbf{C}$ is $\mathbf{T}$-protomodular if $\mathbf{T}$ satisfies Condition \ref{Condition T}.
\end{definition}
The purpose of this section is to prove the following theorem, which in fact says that the notions of $\mathbf{S}$-protomodular and of $\mathbf{T}$-protomodular are equivalent to each other: 
\begin{theorem}\label{Bijection}
	Let $\mathbb{S}$ and $\mathbb{T}$ be the collections of all classes of points and of all classes of generalized points, respectively, in $\mathbf{C}$. Let $F:\mathbb{T}\to\mathbb{S}$ and $G:\mathbb{S}\to\mathbb{T}$ be the maps defined by
	\begin{equation*}
	F(\mathbf{T})=\mathbf{T}\cap\mathbf{Pt}(\mathbf{C}),\,\,\,G(\mathbf{S})=\{(f,g)\in\mathbf{GPt}(\mathbf{C})\,|\,(\pi_2,\langle g,1_C\rangle)\in\mathbf{S}\},
	\end{equation*}
	where $\pi_2$ and $\langle g,1_C\rangle$ are as in the diagram used in Condition 3.2(d). Then $F$ and $G$ induce inverse to each other bijections between the collection of all $\mathbf{T}\in\mathbb{T}$, such that $\mathbf{C}$ is a $\mathbf{T}$-protomodular category, and the collection of all $\mathbf{S}\in\mathbb{S}$, such that $\mathbf{C}$ is an $\mathbf{S}$-protomodular category. 
\end{theorem}
\begin{proof}
	Suppose $\mathbf{S}\in\mathbb{S}$ and $\mathbf{T}\in\mathbb{T}$ satisfy Conditions 3.1 and 3.2, respectively; we have to verify that:
	\begin{itemize}
		\item [(i)] $F(\mathbf{T})$ satisfies Condition 3.1;
		\item [(ii)] $G(\mathbf{S})$ satisfies Condition 3.2;
		\item [(iii)] $GF(\mathbf{T})=\mathbf{T}$;
		\item [(iv)] $FG(\mathbf{S})=\mathbf{S}$.
	\end{itemize}
	Let us give details of these verifications: 
	
(i) is obvious, having in mind that all finite limits in $\mathbf{Pt}(\mathbf{C})$ are component-wise.

(ii), Condition \ref{Condition S}(a): Given $(f:A\to B,g:C\to A)\in G(\mathbf{S})$ and a morphism $x:X\to B$, consider the commutative diagram
\begin{equation*}
\xymatrix@C=3ex@R=5ex{&C\times_BX\ar[dl]\ar@{=}[dd]\ar[rr]&&A\times_BC\times_BX\ar[dl]\ar[dd]\ar[rr]&&C\times_BX\ar[dl]\ar[dd]\\
	C\ar@{=}[dd]\ar[rr]&&A\times_BC\ar[dd]\ar[rr]&&C\ar[dd]\\
	&C\times_BX\ar[dl]\ar[rr]&&A\times_BX\ar[dl]\ar[rr]&&X\ar[dl]^x\\
	C\ar[rr]_g&&A\ar[rr]_f&&B}
\end{equation*}
(with obviously defined unlabeled arrows). In this diagram, since each face of the right-hand cube is a pullback diagram, we have:
\begin{itemize}
\item since $(f,g)$ belongs to $G(\mathbf{S})$, the second row belongs to $\mathbf{S}$;
\item since the second row belongs to $\mathbf{S}$, and $\mathbf{S}$ is pullback stable, the first row also belongs to $\mathbf{S}$;
\item since the first row belongs to $\mathbf{S}$, the third row, which is nothing but the pullback $(f,g)$ along $x$, belongs to $G(\mathbf{S})$.
\end{itemize}

(ii), Condition 3.2(b): Consider the extension $\mathbf{D}$ of $\mathbf{GPt}(\mathbf{C})$ defined simply as the category of all composable pairs of morphisms in $\mathbf{C}$. Define the forgetful functor $U:\mathbf{D}\to\mathbf{C}$ by $U(f:A\to B,g:C\to A)=B$, and, for an object $B$ in $\mathbf{C}$, write $(\mathbf{D}\downarrow B)$ for the fibre of this functor over $B$. $G(\mathbf{S})$ does satisfy Condition 3.2(b) since:
\begin{itemize}
	\item the component-wise limits in $\mathbf{GPt}(\mathbf{C})$ are calculated as in $\mathbf{D}$;
	\item $G(\mathbf{S})$ (obviously) contains the terminal object $(0\to 0,0\to 0)$ of $\mathbf{D}$;
	\item for every object $B$, the fibre inclusion functor $(\mathbf{D}\downarrow B)\to\mathbf{D}$ preserves pullbacks, and the same is true for $\mathbf{Pt}(\mathbf{C})$;
	\item for every morphism $x:X\to B$, the change-of-base functor $x^*:(\mathbf{D}\downarrow B)\to(\mathbf{D}\downarrow X)$ preserves all finite limits, and, again, the same is true for $\mathbf{Pt}(\mathbf{C})$.    
\end{itemize}

(ii), Condition 3.2(c): $G(\mathbf{S})$ does satisfy it by Corollary \ref{Corollary 1}.

(ii), Condition 3.2(d): $G(\mathbf{S})$ does satisfy it since (obviously) the top row of the diagram belongs to $G(\mathbf{S})$ if and only if it belongs to $\mathbf{S}$.

(iii): For a generalized point $(f,g)$, we have
\begin{equation*}
(f,g)\in GF(\mathbf{T})\Leftrightarrow(\pi_2,\langle g,1_C\rangle)\in F(\mathbf{T})\Leftrightarrow(\pi_2,\langle g,1_C\rangle)\in\mathbf{T}\Leftrightarrow (f,g)\in\mathbf{T},
\end{equation*}
where the second equivalence holds because $(\pi_2,\langle g,1_C\rangle)$ automatically belongs to $\mathbf{Pt}(\mathbf{C})$, while the third one holds by Condition \ref{Condition T}(d).

(iv): For a point $(f,g)$, we have 
\begin{equation*}
(f,g)\in FG(\mathbf{S})\Leftrightarrow(f,g)\in G(\mathbf{S})\Leftrightarrow(f,g)\in\mathbf{S},
\end{equation*}
since, when $(f,g)$ is a point, it is the same as $(\pi_2,\langle g,1_C\rangle)$.
\end{proof}
\section{A characterisation of regular Schreier extensions in the category of monoids}

Throughout this section we assume that $\C$ is a category of monoids, for which we will use the additive notation. Since it is a variety of universal algebras, a generalized point in $\C$ is a pair $(f,g)$ for which $fg$ is a surjective homomorphism (of monoids). Let us recall two known definitions (see e.g. \cite{[P]}, \cite{[MMS]}, \cite{[BMMS1]}, \cite{[BMMS2]}, \cite{[MMPS]}): 

\begin{definition}\label{Def Split Schreier} A point (=split generalized point) $(f,g)=$
	$$
	\xymatrix{B\ar[r]^g & A\ar[r]^f & B}
	$$	
(in $\C$) is said to be a Schreier point, if for every $a\in A$	there exists a unique $k\in \mathrm{Ker}(f)$ such that $a=k+gf(a)$.
\end{definition}
\begin{definition}\label{Def Schreier} Let $f:A\rightarrow B$ be a surjective homomorphism of monoids. Then:
\begin{itemize}
	\item[(a)] an element $a$ in $A$ is said to be a representative of an element $b$ in $B$, if $f(a)=b$ and for every $a'\in A$ with $f(a')=b$, there exists a unique element $k\in \mathrm{Ker}(f)$ with $a'=k+a$;
	\item[(b)] $f$ is said to be a Schreier epimorphism if every element of $B$ has a representative;
	\item[(c)] $f$ is said to be a regular Schreier epimorphism if it is a Schreier epimorphism and the set of all representatives of elements of $B$ is a submonoid of $A$.
\end{itemize}	
\end{definition}
Let us introduce:
\begin{definition}\label{def Schreier GP} A generalized point $(f,g)=$
		$$
	\xymatrix{C\ar[r]^g & A\ar[r]^f & B}
	$$	
	(in $\C$) is said to be a Schreier generalized point, if for every $a\in A$ and every $c\in C$ with $f(a)=fg(c)$, there exists a unique $k\in \mathrm{Ker}(f)$ such that $a=k+g(c)$.
\end{definition}
\begin{remark}
	Every Schreier generalized point is obviously strong.
\end{remark}

The following two theorems describe the connection between Definitions \ref{Def Split Schreier}, \ref{Def Schreier}(c), and \ref{def Schreier GP}.
\begin{theorem}\label{Thm Split Schreier, Shcreier, Schreier GP} Let $\mathbf{S}$ and $\mathbf{T}$ be the classes of Schreier points and of Schreier generalized points, respectively. Then:
\begin{itemize}
\item[(a)] $\mathbf{C}$ is an $\mathbf{S}$-protomodular cateogry in the sense of \cite{[BMMS]} (this is proved in \cite{[BMMS2]});
\item[(b)] $\C$ is a $\mathbf{T}$-protomodular category in the sense of our Definition \ref{Def T-protomodular category};
\item[(c)] $\mathbf{S}$ and $\mathbf{T}$ correspond to each other under the bijection described in Theorem \ref{Bijection}.
\end{itemize}
\end{theorem}	
	
\begin{proof}As follows from \ref{Thm Split Schreier, Shcreier, Schreier GP}$(a)$ (which was poved in \cite{[BMMS2]}) and Theorem \ref{Bijection}, all we need to show is that $G(\mathbf{S})=T$. That is, we only need to show that, in the notation above, $(f,g)$ is a Schreier generalized point if and only if $(\pi_2,\langle g,1_C\rangle)=$
	$$
\xymatrix@C=3pc{C\ar[r]^-{\langle g,1_C\rangle} & A\times_BC \ar[r]^-{\pi_2} & C}
	$$
is a Schreier point. However, this is straighforward, having in mind that:
\begin{itemize}
	\item $(a,c)\in A\times C$ belongs to $A\times_BC$ if and only if $a\in A$ and $c\in C$ are such that $f(a)=fg(c)$;
	\item $\mathrm{Ker}(\pi_2)=\mathrm{Ker}(f)\times\{0\}$;
	\item $(a,c)=(k,0)+\langle g,1_C\rangle \pi_2(a,c) \iff (a,c)=(k,0)+(g(c),c)\iff a=k+g(c)$.
\end{itemize}
\end{proof}
	
\begin{theorem}A monoid homomorphism $f:A\rightarrow B$ is a regular Schreier epimorphism if and only if there exists a monoid homomorphism $g:C\rightarrow A$ making $(f,g)$ a Schreier generalized point.
\end{theorem}
\begin{proof}
``If": Suppose $(f,g)$ is a Schreier generalized point. Since $fg$ is surjective, for every $b\in B$ there exists $c\in C$ with $fg(c)=b$, making $g(c)$ a representative of $b$. Therefore, $f$ is a Schreier epimorphism. It remains to prove that if $a$ and $a'$ are representatives of $b$ and $b'$ respectively, then $a+a'$ is a representative (of $b+b'$). As follows from Proposition 3.5 of \cite{[MMPS]}, we can replace $a$ with any other representative $b$, and $a'$ with any other representative of $b'$. We replace $a$ with $g(c)$, and $a'$ with $g(c')$, where $fg(c)=b$ and $fg(c')=b'$, and we know that $g(c)+g(c')=g(c+c')$ is a representative (of $b+b'$).

``Only if": Suppose $f:A\rightarrow B$ is a regular Schreier epimorphism. Let $C$ be the free monoid on $B$, and $\eta:B\rightarrow C$ the canonical inclusion map. For each $b\in B$ choose any representative $a_b$ of $b$, and define $g:C\rightarrow B$ as the unique monoid homomorphism carrying $\eta(b)$ to $a_b$, for each $b\in B$. To prove that $(f,g)$ is a Schreier generalized point it suffices to prove that each $g(c)$ ($c\in C$) is a representative, but this follows from the fact that so are all $g(\eta(b))=a_b$ ($b\in B$) and $C$ is generated by $\eta(B)$.
\end{proof}	

\begin{corollary}
	A monoid homomorphism $f:A\rightarrow B$ is a regular Schreier epimorphism if and only if there exists a monoid homomorphism $g:C\rightarrow A$ such that $fg$ is surjective and $(\pi_2, \langle g,1_C \rangle)=$
		$$
	\xymatrix@C=3pc{C\ar[r]^-{\langle g,1_C\rangle} & A\times_BC \ar[r]^-{\pi_2} & C}
	$$
	is a Schreier point.
\end{corollary}

\end{document}